\documentclass[a4paper,11pt]{article}
\usepackage[T1]{fontenc}
\usepackage{amsfonts}
\usepackage{verbatim}
\usepackage{amsmath, amsthm, amssymb, mathrsfs}

\title{Domains of Attraction for Positive and Discrete Tempered Stable Distributions}
\author{Michael Grabchak\footnote{Email address: mgrabcha@uncc.edu}\\
{\it University of North Carolina Charlotte}}

\begin{document}
\newtheorem{thrm}{Theorem}
\newtheorem{prop}[thrm]{Proposition}
\newtheorem{defn}[thrm]{Definition}
\newtheorem{cor}[thrm]{Corollary}
\newtheorem{lemma}[thrm]{Lemma}
\newtheorem{remark}[thrm]{Remark}

\newcommand{\rd}{\mathrm d}
\newcommand{\rE}{\mathrm E}
\newcommand{\ts}{TS^p_\alpha}
\newcommand{\ets}{ETS^p_\alpha}
\newcommand{\tr}{\mathrm{tr}}
\newcommand{\iid}{\stackrel{\mathrm{iid}}{\sim}}
\newcommand{\eqd}{\stackrel{d}{=}}
\newcommand{\cond}{\stackrel{d}{\rightarrow}}
\newcommand{\conv}{\stackrel{v}{\rightarrow}}
\newcommand{\conw}{\stackrel{w}{\rightarrow}}
\newcommand{\conp}{\stackrel{p}{\rightarrow}}
\newcommand{\confdd}{\stackrel{fdd}{\rightarrow}}
\newcommand{\plim}{\mathop{\mathrm{p\mbox{-}lim}}}
\newcommand{\lgg}{\mathrm{log}}
\newcommand{\dlim}{\operatorname*{d-lim}}

\maketitle
\begin{abstract}
We introduce a large and flexible class of discrete tempered stable distributions, and analyze the domains of attraction for both this class and the related class of positive tempered stable distributions. Our results suggest that these are natural models for sums of independent and identically distributed random variables with tempered heavy tails tails, i.e.\ tails that appear to be heavy up to a point, but ultimately decay faster.
\end{abstract}

\section{Introduction}

Stable distributions play a central role in many applications. However, their use is limited by the fact that they have an infinite variance, which is not realistic for most real-world applications. This has led to the development of tempered stable distributions, which is a class of models obtained by modifying the tails of stable distributions to make them lighter, while leaving their central portions, essentially, unchanged. Perhaps the earliest models of this type are Tweedie distributions, which were introduced in the seminal paper Tweedie (1984) \cite{Tweedie:1984}.
A more general approach, allowing for a wide variety of tail behavior, is given in Rosi\'nski (2007) \cite{Rosinski:2007}. That approach was further generalized in several directions in \cite{Rosinski:Sinclair:2010}, \cite{Bianchi:Rachev:Kim:Fabozzi:2011}, and \cite{Grabchak:2012}. A survey, along with a historical overview and many references can be found in \cite{Grabchak:2016}. We will focus on the class of positive tempered stable (PTS) distributions. This class is important for many applications including actuarial science \cite{Griffin:Maller:Roberts:2013}, biostatistics \cite{Palmer:Ridout:Morgan:2008}, mathematical finance \cite{Valdivieso:Schoutens:Tuerlinckx:2009}, and computer science  \cite{Cao:Grabchak:2014}. 

In a different direction, stable distributions have been modified to deal with over-dispersion when modeling count data. Specifically, the class of discrete stable distributions was introduced in Steutel and van Harn (1979) \cite{Steutel:vanHarn:1979}, see also \cite{Christoph:Schreiber:1998}, \cite{Steutel:vanHarn:2004}, and \cite{Klebanov:Slamova:2013}. As with continuous stable distributions, these models have an infinite variance, which has led to the development of a tempered modification. In particular, \cite{Hougaard:1987} introduced a class of models that has come to be known as Poisson-Tweedie. The name comes from the fact that these can be represented as a Poisson process subordinated by a Tweedie distribution. Many results along with applications to a variety of areas including economics, biostatistics, bibliometrics, and ecology can be found in, e.g.\  \cite{Hougaard:Lee:Whitmore:1997}, \cite{Zhu:Joe:2009}, \cite{El-Shaarawi:Zhu:Joe:2010}, \cite{Barndorff-Nielsen:Pollard:Shephard:2012}, \cite{Jorgensen:Kokonendji:2016}, \cite{Baccini:Barabesi:Stracqualursi:2016}, and the references therein. 

In this paper, we introduce a large class of discrete tempered stable (DTS) distributions, which generalize the class of Poisson-Tweedie models. We then prove limit theorems for PTS and DTS distributions. Just as generalizations of the central limit theorem explain how stable and discrete stable distributions approximate sums of independent and identically distributed (iid) random variables with heavy tails, our theorems aim to provide a theoretical justification for the use of PTS and DTS models in approximating sums of iid random variables with tempered heavy tails, i.e.\ tails that appear to be heavy up to point, but have been modified to, ultimately, decay faster. For a discussion of how such models occur in practice see \cite{Grabchak:Samorodnitsky:2010} and \cite{Cao:Grabchak:2014}. Related limit theorems for Poisson-Tweedie distributions are given in \cite{Jorgensen:Kokonendji:2016}. In the continuous case, similar results for Tweedie distributions were studied in \cite{Grabchak:Molchanov:2015} and, from a different perspective, convergence of certain random walks to tempered stable distributions were studied in \cite{Chakrabarty:Meerschaert:2011}.

Before proceeding we introduce some notation. We write $\mathbb N=\{1,2,\dots\}$, $\mathbb Z_+=\mathbb N\cup\{0\}$,  $\mathbb R_+=[0,\infty)$, and $\mathfrak B(\mathbb R_+)$ to denote the Borel sets on $\mathbb R_+$. For a probability measure $\mu$ with support contained in $\mathbb R_+$ we write $\hat\mu(z)=\int_{\mathbb R_+} e^{-zx}\mu(\rd x)$ to denote its Laplace transform and $X\sim\mu$ to denote that $X$ is a random variable with distribution $\mu$. For a function $f:\mathbb R_+\mapsto\mathbb R_+$ and $\beta\in\mathbb R$, we write $f\in RV_\beta$ to denote that $f$ is regularly varying with index $\beta$, i.e.\ that
$$
\lim_{t\to\infty} \frac{f(xt)}{f(t)}=x^\beta \ \ \mbox{for any } x>0.
$$
We write $1_A$ to denote the indicator function on set $A$, for $x>0$ we write $\Gamma(x) = \int_0^\infty e^{-t} t^{x-1}\rd t$ to denote the gamma function, and we write $\conp$, $\cond$, $\conw$, and $\eqd$ to denote, respectively, convergence in probability, convergence in distribution, weak convergence, and equality in distribution. For $c\in[-\infty,\infty]$, we write $f(t)\sim g(t)$ as $t\to c$ to denote that $\lim_{t\to c} \frac{f(x)}{g(x)}=1$.

\section{Positive Stable and Positive Tempered Stable Distributions}

In this section we formally introduce positive stable and positive tempered stable distributions. We begin by recalling some basic facts about positive infinitely divisible distributions. An infinitely divisible distribution $\mu$, with support contained in $\mathbb R_+$, has a Laplace transform of the form
\begin{eqnarray}\label{eq: Laplace ID}
\hat\mu(z) = 
\exp\left\{-bz - \int_{(0,\infty)}\left(1-e^{-z x}\right)M(\rd x)\right\}, \ \ z\ge0,
\end{eqnarray}
where $b\ge0$ and $M$ is a Borel measure on $(0,\infty)$ satisfying
\begin{eqnarray}\label{eq: finite var}
\int_{(0,\infty)} \left(x\wedge1\right)M(\rd x)<\infty.
\end{eqnarray}
Here, $b$ is called the drift and $M$ is called the L\'evy measure. These parameters uniquely determine the distribution and we write $\mu=ID_+(M,b)$. For a general reference on infinitely divisible distributions see \cite{Sato:1999}.

A probability measure $\mu$ on $\mathbb R_+$ is said to be strictly $\alpha$-stable if, for any $n\in\mathbb N$ and $X_1,X_2,\dots,X_n\iid\mu$, we have 
\begin{eqnarray}\label{eq: defn pos stable}
X_1\eqd n^{-1/\alpha}\left(X_1+X_2+\cdots+X_n\right).
\end{eqnarray}
By positivity, we necessarily have $\alpha\in(0,1]$. If $\alpha\in(0,1)$, then $\mu=ID_+(M_\alpha,0)$, where
$$
M_\alpha(\rd x) = \eta x^{-1-\alpha}1_{x>0}\rd x
$$
for some $\eta\ge0$. If $\alpha=1$ then $\mu=ID_+(0,\eta)$ for some $\eta\ge0$, and thus $\mu$ is a point mass at $\eta$. For $\alpha\in(0,1)$ the Laplace transform is of the form
$$
\hat\mu(z)=e^{-\eta \frac{\Gamma(1-\alpha)}{\alpha}z^\alpha}, \ \ z\ge0.
$$
We denote this distribution by $PS_\alpha(\eta)$. Note that $PS_\alpha(0)$ is a point mass at zero for all $\alpha\in(0,1)$. For more about stable distributions on $\mathbb R_+$ see \cite{Steutel:vanHarn:2004}.

It is well-known that, for $\alpha\in(0,1)$, stable distributions have an infinite mean, which is not realistic for many applications. This has lead to the development of distributions that look stable-like in some large central region, but with lighter tails. Following \cite{Rosinski:Sinclair:2010}, we define positive tempered stable distributions as follows.

\begin{defn}\label{defn: PTS}
A distribution $\mu=ID_+(M,b)$ is called a positive tempered stable (PTS) distribution if $b\ge0$ and
$$
M(\rd x) = \eta q(x)x^{-1-\alpha}1_{x>0}\rd x,
$$
where $\alpha\in(0,1)$, $\eta\ge0$, and $q:\mathbb R_+\mapsto \mathbb R_+$ is a bounded, non-negative, Borel function with $\lim_{x\downarrow0}q(x) = 1$, and satisfying
\begin{eqnarray}\label{eq: integ on q}
\int_0^\infty \left(1\wedge x\right)q(x)x^{-1-\alpha}\rd x <\infty.
\end{eqnarray}
We call $q$ the tempering function and we write $\mu=PTS_\alpha(q,\eta,b)$. When $b=0$, we write $PTS_\alpha(q,\eta)=PTS_\alpha(q,\eta,0)$.
\end{defn}

We are motivated by the case where the tempering function, $q$, satisfies the additional condition that $\lim_{x\to\infty} q(x) = 0$. In this case, $PTS_\alpha(q,\eta)$ is similar to $PS_\alpha(\eta)$ in some central region, but with lighter tails. In this sense, $q$ ``tempers'' the tails of the stable distribution. Despite this motivation, none of the results of the paper require this additional condition. We now give several examples of  tempering functions, others can be found in, e.g.\ \cite{Terdik:Woyczynski:2006} and \cite{Grabchak:2016}.\\

\noindent\textbf{Examples.} \textbf{1.} When $q\equiv 1$ there is no tempering and  $PTS_\alpha(q,\eta)=PS_\alpha(\eta)$. 
\textbf{2.} When $q(x) = e^{-ax}$ for some $a>0$, we get the class of Tweedie distributions, which were introduced in \cite{Tweedie:1984}. When $\alpha=.5$, these correspond to inverse Gaussian distributions, see e.g.\ \cite{Seshadri:1993}. 
\textbf{3.} When $q(x) =1_{[0\le x<a]}$ for some $a>0$, we call this truncation. Such distributions are important for certain limit theorems, see \cite{Chakrabarty:Samorodnitsky:2012}.

\section{Discrete Stable and Discrete Tempered Stable Distributions}

A discrete analogue of stable distributions was introduced in \cite{Steutel:vanHarn:1979}.  Here \eqref{eq: defn pos stable} is modified to ensure that the right side remains an integer. Specifically, \cite{Steutel:vanHarn:1979} introduced the so-called `thinning' operation $\circ$, which is defined as follows. If $\gamma\in[0,1]$ and $X$ is a random variable with support contained in $\mathbb Z_+$, then $\gamma\circ X$ is a random variable with distribution
$$
\gamma\circ X \eqd \sum_{i=1}^X \epsilon_i,
$$
where $\epsilon_1,\epsilon_2,\dots$ are iid random variables independent of $X$ having a Bernoulli distribution with $P(\epsilon_i=1) = 1-P(\epsilon_i=0)=\gamma$. Here and throughout, we set $\sum_{i=1}^0 \epsilon_i =0$. Note that, if $P(s)$ is the probability generating function (pgf) of $X$, i.e.\ $P(s)=\rE[s^X]$, then the pgf of $\gamma\circ X$ is $P(1-\gamma+\gamma s)$.

For $\alpha\in(0,1]$, a distribution $\mu$ on $\mathbb Z_+$ is called discrete $\alpha$-stable if for any $n\in\mathbb N$ we have
$$
X_1 \eqd n^{-1/\alpha} \circ \left(X_1+X_2+\cdots+X_n\right),
$$
where $X_1,X_2,\dots\iid \mu$. The class of discrete $1$-stable distributions coincides with the class of Poisson distributions. For $\alpha\in(0,1)$ the pgf of a discrete stable distribution is of the form
$$
\int_{\mathbb Z_+}s^x\mu(\rd x)= e^{-\eta\frac{\Gamma(1-\alpha)}{\alpha}(1-s)^\alpha}, \ \ |s|\le1,
$$
where $\eta\ge0$ is a parameter. We denote this distribution by $DS_\alpha(\eta)$. A useful representation of discrete stable distributions is given in Theorem 6.7 on page 371 of \cite{Steutel:vanHarn:2004}. It is as follows.

\begin{prop}
Fix $\alpha\in(0,1)$ and $\eta\ge0$. If $\{N_t:t\ge0\}$ is a Poisson process with rate $1$ and $T\sim PS_\alpha(\eta)$ is independent of this process, then $N_{T}\sim DS_\alpha(\eta)$.
\end{prop}

By analogy, we define discrete tempered stable distributions as follows. 

\begin{defn}\label{defn: DTS}
Fix $\alpha\in(0,1)$ and $\eta\ge0$. Let $T\sim PTS_\alpha(q,\eta)$ and let $\{N_t:t\ge0\}$ be a Poisson process with rate $1$ independent of $T$. The distribution of $N_{T}$ is called a discrete tempered stable (DTS) distribution. We denote this distribution by $DTS_\alpha(q,\eta)$.
\end{defn}

By a simple conditioning argument, the pgf of $DTS_\alpha(q,\eta)$ is, for $s\in(0,1]$
\begin{eqnarray}\label{eq: pgf}
\rE[s^{N_T}] =\rE[e^{-(1-s)T}]=\exp\left\{- \eta\int_{(0,\infty)}\left(1-e^{-(1-s)x}\right)q(x)x^{-1-\alpha}\rd x\right\}.
\end{eqnarray}

\begin{remark}
There are two simple ways to generalize Definition \ref{defn: DTS}. The first is to allow the rate of the Poisson process to be $r>0$ not necessarily $1$. However, in this case, the distribution of $N_T$ is $DTS_\alpha(q_r,r^{\alpha}\eta)$, where $q_r(x) = q(x/r)$. The second is to allow $T\sim PTS_\alpha(q,\eta,b)$ with $b>0$. In this case, the distribution of $N_T$ is the convolution of $DTS_\alpha(q,\eta)$ and a Poisson distribution with mean $b$.
\end{remark}

We can consider the same tempering functions as for PTS distributions. This leads to the following examples.\\

\noindent \textbf{Examples.} \textbf{1.} When $q\equiv 1$ we have  $DTS_\alpha(q,\eta)=DS_\alpha(\eta)$. 
\textbf{2.} When $q(x) = e^{-ax}$ for $a>0$ the corresponding distributions are Poisson-Tweedie. When $\alpha=.5$ these correspond to Poisson Inverse Gaussian distributions, which were introduced in \cite{Holla:1967}. 
\textbf{3.} When $q(x) = 1_{0\le x<a}$  for $a>0$, we are in the case of truncation.\\

We conclude this section by showing that we can approximate PTS distributions by DTS distributions.  The idea is motivated by \cite{Klebanov:Slamova:2013}, which gives similar results for certain generalizations of discrete stable distributions. Let $q$ be a tempering function. For any $a>0$ define $X_a \sim DTS_\alpha(q_{1/a},a^{-\alpha}\eta)$, where $q_{1/a}(x) = q(ax)$. Since $X_a$ is defined on $\mathbb Z_+$, $aX_a$ is defined on $a\mathbb Z_+=\{0,a,2a,\dots\}$. 

\begin{prop}
We have
$$
aX_a \cond PTS_\alpha(q,\eta) \ \ \mbox{as } a\downarrow0.
$$
\end{prop}

\begin{proof}
From \eqref{eq: pgf} it follows that the Laplace transform of $aX_a$ is given, for $z\ge0$,  by
\begin{eqnarray*}
\rE[e^{-z aX_a}] 
 &=& \exp\left\{-a^{-\alpha} \eta\int_{(0,\infty)}\left(1-e^{-(1-e^{-za})x}\right)q(ax)x^{-1-\alpha}\rd x\right\}\\
&=&\exp\left\{-\eta\int_{(0,\infty)}\left(1-e^{-\frac{(1-e^{-za})}{a}x}\right)q(x)x^{-1-\alpha}\rd x\right\}\\
&\to& \exp\left\{-\eta\int_{(0,\infty)}\left(1-e^{-zx}\right)q(x)x^{-1-\alpha}\rd x\right\},
\end{eqnarray*}
as $a\downarrow0$. Here the convergence follows by the facts that $\frac{(1-e^{-za})}{a}\to z$, $\left(1-e^{-\frac{(1-e^{-za})}{a}x}\right)\le 1\wedge\frac{(1-e^{-za})x}{a}\le 1\wedge (zx)\le(z+1)(1\wedge x)$, and dominated convergence.
\end{proof}

\section{Main Results}

Let $\mu$ be a probability measure on $\mathbb R_+$ such that, for $t>0$,
\begin{eqnarray}\label{eq: mu}
\mu(\{x:x>t\}) = t^{-\alpha} L(t)
\end{eqnarray}
for some $\alpha\in(0,1)$ and $L\in RV_0$.  Let
\begin{eqnarray}\label{eq: at}
V(t) = t^\alpha/L(t) \mbox{ and } a_t = 1/V^\leftarrow(t),
\end{eqnarray}
where $V^\leftarrow(t) = \inf\{s: V(s)>t\}$ is the generalized inverse of $V$, satisfying   
$$
V(V^\leftarrow(t)) \sim V^\leftarrow(V(t)) \sim t \mbox{ as }t\to\infty,
$$
see \cite{Bingham:Goldie:Teugels:1987}. Note that $a_t \in RV_{-1/\alpha}$ and thus that $a_n\to0$ as $n\to\infty$. The following lemma is well-known, but, for completeness, its proof is given in Section \ref{sec: proofs}.

\begin{lemma}\label{lemma: DOA stable}
If $X_1,X_2,\dots\iid \mu$ then 
$$
a_n\sum_{i=1}^n X_i \cond PS_\alpha(\alpha).
$$
\end{lemma}

We now consider the effect of tempering on this result. Let $q$ be a tempering function and, for $\ell>0$, define
$$
q_\ell(x) = q(x/\ell) \mbox{  and  } \mu_\ell(\rd x) = c_\ell q_\ell(x) \mu(\rd x),
$$
where
$$
c_\ell = \left[\int_{[0,\infty)}q_\ell(x) \mu(\rd x)\right]^{-1}
$$
is a normalizing constant. Note that, as $\ell\to\infty$, we have $c_\ell\to1$ and $\mu_\ell\conw\mu$. Thus, for large $\ell$, $\mu_\ell$ is close to $\mu$ in some central region, but, if $\lim_{x\to\infty}q_\ell(x)=0$, then it has lighter tails. In this sense, we interpret $\mu_\ell$ as a tempered version of $\mu$.\\

\noindent \textbf{Examples.} \textbf{1.} When $q\equiv 1$ there is no tempering and $\mu_\ell=\mu$ for each $\ell>0$. 
\textbf{2.} When $q(x) = e^{-ax}$ for some $a>0$, we have $q_\ell(x) = e^{-ax/\ell}$. Thus, $\mu_\ell$ is an Esscher transform of $\mu$. 
\textbf{3.} When $q(x) = 1_{0\le x<a}$  for some $a>0$, we have $q_\ell(x) = 1_{0\le  x<a\ell}$. Thus, $\mu_\ell$ is $\mu$ truncated at $a\ell$. This means that, if $X\sim\mu$, then $\mu_\ell$ is the conditional distribution of $X$ given the event $[X<a\ell]$.\\

Examples 2 and 3 above lead to different modifications of $\mu$ which, for large values of $\ell$, are similar to $\mu$ in some central portion, but have lighter tails. We now give our main result for convergence to PTS distributions.

\begin{thrm}\label{thrm: main for pos}
Let $\{\ell_n\}$ be a sequence of positive numbers with $\ell_n\to\infty$, let $X_{n1},X_{n2},\dots,X_{nn}\iid \mu_{\ell_n}$ for each $n\in\mathbb N$, and let $D$ be the set of discontinuities of $q$. Assume that Lebesgue measure of $D$ is $0$. If $a_n\ell_n\to c\in(0,\infty)$, then
$$
a_n\sum_{i=1}^n X_{ni}\conp PTS_\alpha(q_c,\alpha),
$$
where $q_c(x) =  q(x/c)$ for $x\ge0$. If $a_n\ell_n\to\infty$, then
\begin{eqnarray}\label{eq: to PS}
a_n\sum_{i=1}^n X_{ni}\conp PS_\alpha(\beta)
\end{eqnarray}
with $\beta=\alpha$. If $a_n\ell_n\to0$ and $\lim_{x\to\infty}q(x)=\zeta<\infty$, then \eqref{eq: to PS} holds with $\beta=\alpha\zeta$.
\end{thrm}

\begin{proof}
The proof can be found in Section \ref{sec: proofs}.
\end{proof}

\begin{remark} 
For most applications the parameter $\ell$ is not actually approaching infinity.  Instead, it is some fixed but (very) large constant. Since $a_\bullet\in RV_{-1/\alpha}$, we can write $a_n\ell = [n^{-1}\ell^{\alpha}]^{1/\alpha}L'(n)$ for some $L'\in RV_0$. Now, consider the sum of $n$ iid random variables from $\mu_\ell$, and assume that the tempering function $q$
is such that $\mu_\ell$ has a finite variance. Theorem \ref{thrm: main for pos} can be interpreted as follows. When $n$ is on the order of $\ell^{\alpha}$ the distribution of the sum is close to $PTS_\alpha(q_c,\alpha)$. However, once $n$ is much larger than $\ell^{\alpha}$, the central limit theorem will take effect and the distribution of the sum will be well approximated by the Gaussian.  A constant that determines when such regimes occur was called the ``natural scale'' in \cite{Grabchak:Samorodnitsky:2010}. Thus, in this case, the natural scale is $\ell^\alpha$. Using slightly different perspectives, this was previously found to be the natural scale for Tweedie distributions in \cite{Grabchak:Samorodnitsky:2010} and \cite{Grabchak:Molchanov:2015}.
\end{remark}

The following transfer lemma allows us to transfer convergence results from the case of multiplicative scaling to that of scaling using the thinning operation $\circ$. It is an extension of a remark in \cite{Steutel:vanHarn:1979}.

\begin{lemma}\label{lemma: transfer}
Let $\{X_n\}$ be a sequence of random variables on $\mathbb Z_+$ and assume that $\{\gamma_n\}$ is a deterministic sequence in $[0,1]$ with $\gamma_n\to0$. If $\gamma_n X_n\cond X$ for some random variable $X$, then
$$
\gamma_n\circ X_n \cond N_X,
$$
where $\{N_t:t\ge0\}$ is a Poisson process with rate $1$ and independent of $X$.
\end{lemma}

\begin{proof}
The proof can be found in Section \ref{sec: proofs}.
\end{proof}

Combining this with Lemma \ref{lemma: DOA stable} gives the following.

\begin{lemma}\label{lemma: DOA disc stable}
Assume that the support of $\mu$ is contained in $\mathbb Z_+$. If $X_1,X_2,\dots\iid \mu$ then 
$$
a_n\circ\sum_{i=1}^n X_i \cond DS_\alpha(\alpha).
$$
\end{lemma}

Now combining Theorem \ref{thrm: main for pos} with Lemma \ref{lemma: transfer} gives our main result for convergence to DTS distributions.

\begin{thrm}\label{thrm: main for discrete}
Assume that the support of $\mu$ is contained in $\mathbb Z_+$. Let $\{\ell_n\}$ be a sequence of positive numbers with $\ell_n\to\infty$, let $X_{n1},X_{n2},\dots,X_{nn}\iid \mu_{\ell_n}$ for each $n\in\mathbb N$, and let $D$ be the set of discontinuities of $q$. Assume that Lebesgue measure of $D$ is $0$. If $a_n\ell_n\to c\in(0,\infty)$, then
$$
a_n\circ\sum_{i=1}^n X_{ni}\conp DTS_\alpha(q_c,\alpha),
$$
where $q_c(x) = q(x/c)$ for $x\ge0$. If $a_n\ell_n\to\infty$, then
\begin{eqnarray}\label{eq: to DS}
a_n\circ\sum_{i=1}^n X_{ni}\conp DS_\alpha(\beta)
\end{eqnarray}
with $\beta=\alpha$. If $a_n\ell_n\to0$ and $\lim_{x\to\infty}q(x)=\zeta<\infty$, then \eqref{eq: to DS} holds with $\beta=\alpha\zeta$.
\end{thrm}

\section{Proofs}\label{sec: proofs}

The proofs of Lemma \ref{lemma: DOA stable} and Theorem \ref{thrm: main for pos} are based on verifying conditions for the convergence of sums of triangular array.  The general theory can be found in, e.g.\ \cite{Meerschaert:Scheffler:2001} or \cite{Kallenberg:2002}. However, for the situations considered here, the conditions can be simplified. These are as follows.

\begin{prop}\label{prop: conv of sums of arrays}
Let $k_n$ be a sequence of positive integers with $k_n\to\infty$, let $M$ be  a Borel measure on $(0,\infty)$ satisfying \eqref{eq: finite var}, and let $\{X_{nm}:n=1,2,\dots,m=1,2,\dots,k_n\}$ be nonnegative random variables such that, for every $n$, the random variables $X_{n1},X_{n2},\dots,X_{nk_n}$ are iid and $X_{n1}\conp 0$ as $n\to\infty$.  If, for every $s>0$ with $M(\{s\})=0$, we have
\begin{eqnarray}\label{eq: conv to Levy meas}
\lim_{n\to\infty}k_nP\left(X_{n1}>s\right) = M((s,\infty))
\end{eqnarray}
and
\begin{eqnarray}\label{eq: conv of drift to zero}
\lim_{\epsilon\downarrow0}\limsup_{n\to\infty} k_n\rE\left[X_{n1}1_{[X_{n1}<\epsilon]}\right]=0
\end{eqnarray}
then
$$
\sum_{m=1}^{k_n} X_{nm} \cond ID_+(M,0).
$$
\end{prop}

\begin{proof}
Let $\nu_n$ be the distribution of $X_{n1}$, let $\nu= ID_+(M,0)$, and let $\hat\nu_n(u)=\int_{[0,\infty)} e^{-ux} \nu_n(\rd x)$ and $\hat\nu(u)=\int_{[0,\infty)} e^{-ux} \nu(\rd x)$ be the Laplace transforms of $\nu_n$ and $\nu$ respectively. The Laplace transform of the distribution of $\sum_{m=1}^{k_n} X_{nm}$ is $\left[\hat\nu_n(u)\right]^{k_n}$. We must show that 
$$
\lim_{n\to\infty}\left[\hat\nu_n(u)\right]^{k_n}= \hat\nu(u), \ \ \ u\ge0.
$$
We will write the left side in a simpler form. Specifically, we have
\begin{eqnarray*}
\lim_{n\to\infty}\left[\hat\nu_n(u)\right]^{k_n} &=& \lim_{n\to\infty}\exp\left\{k_n\log\left[\hat\nu_n(u)\right]\right\}\\
&=& \lim_{n\to\infty}\exp\left\{k_n\left[\hat\nu_n(u)-1\right]\right\}\\
&=&\lim_{n\to\infty}\exp\left(-k_n\int_{[0,\infty)}\left(1-e^{-ux}\right)\nu_n(\rd x)\right),
\end{eqnarray*}
where the second equality follows from the facts that $\log(x)\sim(x-1)$ as $x\to1$ and $\lim_{n\to\infty}\hat\nu_n(u)=1$ for each $u\ge0$ since $X_{n1}\conp 0$ as $n\to\infty$. 

Since, for fixed $u$, $f_u(x) = \left(1-e^{-ux}\right)$ is a bounded and continuous function of $x$, by the Portmanteau Theorem for vague convergence (see Theorem 1 in \cite{Barczy:Pap:2006}) \eqref{eq: conv to Levy meas} implies that, for any $\epsilon>0$,
$$
\lim_{n\to\infty}k_n\int_{[\epsilon,\infty)}\left(1-e^{-ux}\right)\nu_n(\rd x) = \int_{[\epsilon,\infty)}\left(1-e^{-ux}\right)M(\rd x).
$$
By \eqref{eq: conv of drift to zero} and well-known facts about the exponential function, we have
\begin{eqnarray*}
0&\le&\lim_{\epsilon\downarrow0}\liminf_{n\to\infty}k_n\int_{[0,\epsilon)}\left(1-e^{-ux}\right)\nu_n(\rd x)\\
&\le& \lim_{\epsilon\downarrow0}\limsup_{n\to\infty}k_n\int_{[0,\epsilon)}\left(1-e^{-ux}\right)\nu_n(\rd x)\\
&\le& \lim_{\epsilon\downarrow0}\limsup_{n\to\infty}uk_n\int_{[0,\epsilon)}x\nu_n(\rd x)=0.
\end{eqnarray*}
Combining the above with Lebesgue's dominated convergence theorem gives
\begin{eqnarray*}
&&\liminf_{n\to\infty}k_n\int_{[0,\infty)}\left(1-e^{-ux}\right)\nu_n(\rd x)\\
&&\qquad= \lim_{\epsilon\downarrow0}\liminf_{n\to\infty}k_n\int_{[0,\epsilon)}\left(1-e^{-ux}\right)\nu_n(\rd x) \\
&&\qquad\qquad
+ \lim_{\epsilon\downarrow0}\liminf_{n\to\infty}k_n\int_{[\epsilon,\infty)}\left(1-e^{-ux}\right)\nu_n(\rd x) \\
&&\qquad = \lim_{\epsilon\downarrow0} \int_{[\epsilon,\infty)}\left(1-e^{-ux}\right)M(\rd x)\\
&&\qquad =  \int_{(0,\infty)}\left(1-e^{-ux}\right)M(\rd x).
\end{eqnarray*}
Similarly, we can repeat the above with $\limsup$ in place of $\liminf$. Then, putting everything together gives
\begin{eqnarray*}
\lim_{n\to\infty}\left[\hat\nu_n(u)\right]^{k_n}
&=&\lim_{n\to\infty}\exp\left(-k_n\int_{[0,\infty)}\left(1-e^{-ux}\right)\nu_n(\rd x)\right)\\
&=&\exp\left(-\int_{(0,\infty)}\left(1-e^{-ux}\right)M(\rd x)\right),
\end{eqnarray*}
which is the Laplace transform of $ID_+(M,0)$ as required.
\end{proof}

Before proceeding, we define the Borel measures
\begin{eqnarray}\label{eq: Mn}
M_n(A) = n\int_{[0,\infty)} 1_A(a_n x)\mu(\rd x), \qquad A\in\mathfrak B(\mathbb R_+)
\end{eqnarray}
and 
\begin{eqnarray}\label{eq: Minf}
M_\infty(A) =  \int_{[0,\infty)} 1_A(x) \alpha x^{-\alpha-1}\rd x, \qquad A\in\mathfrak B(\mathbb R_+).
\end{eqnarray}
Note that $M_\infty$ is the L\'evy measure of the distribution $PS_\alpha(\alpha)$.

\begin{lemma}\label{lemma: conv}
The following hold
\begin{eqnarray*}
\lim_{n\to\infty} M_n((s,\infty)) = M_\infty((s,\infty)), \ \ s>0
 \end{eqnarray*}
and
\begin{eqnarray*}
\lim_{\epsilon\downarrow0}\limsup_{n\to\infty}\int_{[0,\epsilon)} xM_n(\rd x) = 0.
 \end{eqnarray*}
\end{lemma}

\begin{proof}
We have, for $s>0$,
\begin{eqnarray*}
\lim_{n\to\infty} M_n((s,\infty))&=&\lim_{n\to\infty} V\left(V^\leftarrow(n)\right)\mu((s/a_n,\infty))\\
&=&\lim_{n\to\infty} V(1/a_n)\frac{1}{V(s/a_n)} = s^{-\alpha} = M_\infty((s,\infty))
 \end{eqnarray*}
and, recalling that $\alpha\in(0,1)$ gives
\begin{eqnarray*}
 \lim_{\epsilon\downarrow0}\limsup_{n\to\infty}\int_{[0,\epsilon)} xM_n(\rd x)
&=&  \lim_{\epsilon\downarrow0}\limsup_{n\to\infty} na_n\int_{[0,\epsilon/a_n)} x\mu(\rd x)  \\
&=&\lim_{\epsilon\downarrow0}\limsup_{n\to\infty} V(1/a_n)a_n\int_{[0,\epsilon/a_n)} x\mu(\rd x)\\
 &=& \lim_{\epsilon\downarrow0}\limsup_{n\to\infty} \epsilon^{-\alpha}V(\epsilon/a_n)a_n\int_{[0,\epsilon/a_n)} x\mu(\rd x)\\
 &=&\lim_{\epsilon\downarrow0}\epsilon^{1-\alpha}\limsup_{n\to\infty}\frac{\int_{[0,\epsilon/a_n)} x\mu(\rd x)}{(\epsilon/a_n)\int_{(\epsilon/a_n,\infty)}\mu(\rd x)}\\
 &=&\lim_{\epsilon\downarrow0}\epsilon^{1-\alpha}\frac{\alpha}{1-\alpha}=0,
\end{eqnarray*}
where the first convergence follows by Theorem 2 on page 283 of \cite{Feller:1971}. 
\end{proof}

\begin{proof}[Proof of Lemma \ref{lemma: DOA stable}.]
Since $a_n\to0$, Slutsky's Theorem implies that $a_nX_1\conp0$. From here, the result follows by combining Lemma \ref{lemma: conv} with Proposition \ref{prop: conv of sums of arrays}.
\end{proof}

\begin{lemma}\label{lemma: hn}
Let $h,h_1,h_2,\dots$ be is a sequence of Borel functions and let $D$ be a Borel set with Lebesgue measure zero such that for any $x\in D^c$ and any sequence of real numbers $x_1,x_2,\dots$ with $x_n\to x$ we have $h_n(x_n)\to h(x)$. Then, for any $s>0$,
$$
\lim_{n\to\infty}\int_{(s,\infty)} h_n(x) M_n(\rd x) = \int_{(s,\infty)} h(x) M_\infty(\rd x).
$$
\end{lemma}

\begin{proof}
Fix $s>0$. Let $m_n=M_n((s,\infty))$, $m_\infty=M_\infty((s,\infty))$, and define the probability measures $M_n^{(s)}(\rd x) = m_n^{-1}1_{x>s} M_n(\rd x)$ and $M_\infty^{(s)}(\rd x) = m_\infty^{-1} 1_{x>s} M_\infty(\rd x)$. From Lemma \ref{lemma: conv} and the Portmanteau Theorem it follows that $M_n^{(s)}\conw M_\infty^{(s)}$. Further, since $M_\infty^{(s)}$ is absolutely continuous with respect to Lebesgue measure, it follows that $M_\infty^{(s)}(D)=0$. From here, a standard result about weak convergence, see e.g.\ Example 32 on page 58 in \cite{Pollard:1984}, implies that
$$
\lim_{n\to\infty}\int_{(s,\infty)} h_n(x) M^{(s)}_n(\rd x) = \int_{(s,\infty)} h(x) M^{(s)}_\infty(\rd x).
$$
The result follows by combining this with the fact that $m_n\to m_\infty$.
\end{proof}

\begin{proof}[Proof of Theorem \ref{thrm: main for pos}]
The proof is based on verifying that the assumptions of Proposition \ref{prop: conv of sums of arrays} hold. Toward this end, note that
\begin{eqnarray*}
nP(a_nX_{1n}>s) &=& nc_{\ell_n} \int_{(s/a_n,\infty)} q_{\ell_n}(x) \mu(\rd x)\\
&\sim& n \int_{(s/a_n,\infty)} q(x/\ell_n) \mu(\rd x) \\
&=& \int_{(s,\infty)} q(x/(a_n\ell_n)) M_n(\rd x).
\end{eqnarray*}
By Lemma \ref{lemma: hn} this converges to
$$
\int_{(s,\infty)} q(x/c) M_\infty(\rd x) = \alpha\int_{(s,\infty)} q_c(x) x^{-1-\alpha}\rd x,
$$
where we interpret $q(x/c)=1$ if $c=\infty$ and $q(x/c)=\zeta$ if $c=0$. Further,
\begin{eqnarray*}
&&\lim_{\epsilon\downarrow0}\limsup_{n\to\infty} n\rE\left[a_nX_{1}1_{[a_nX_{1}<\epsilon]}\right]\\
&&\qquad=\lim_{\epsilon\downarrow0}\limsup_{n\to\infty} na_n c_{\ell_n}\int_{[0,\epsilon/a_n)} x q(x/\ell_n)\mu(\rd x)\\
&&\qquad= \lim_{\epsilon\downarrow0}\limsup_{n\to\infty} na_n \int_{[0,\epsilon/a_n)} x q(x/\ell_n)\mu(\rd x)\\
&&\qquad\le K\lim_{\epsilon\downarrow0}\limsup_{n\to\infty} \int_{[0,\epsilon)} x M_n(\rd x)=0,
\end{eqnarray*}
where the last line follows by Lemma \ref{lemma: conv} and $K$ is any upper bound on $q$.
\end{proof}

\begin{proof}[Proof of Lemma \ref{lemma: transfer}]
First note that, by Slutsky's Theorem, for any $t>0$
$$
-X_n \log(1-\gamma_n t) = -\gamma_n t X_n \frac{\log(1-\gamma_n t)}{\gamma_n t}\cond X t.
$$
Let $P_n$ be the pgf of the distribution of $X_n$. The pgf of the distribution of $\gamma_n\circ X_n$ is then $P_n(1-\gamma_n+\gamma_n s)=P_n(1-\gamma_n t)$, where $t=1-s$. Since convergence in distribution implies convergence of Laplace transforms,
\begin{eqnarray*}
\lim_{n\to\infty}P_n(1-\gamma_n t) = \lim_{n\to\infty}\rE\left[e^{X_n\log(1-\gamma_n t)}\right] = \rE\left[e^{-X t}\right] =  \rE\left[e^{-X (1-s)}\right].
\end{eqnarray*}
Observing that
$$
\rE[s^{N_X}] = \rE\left[\rE[s^{N_X}|X]\right] = \rE[e^{-X(1-s)}]
$$
gives the result.
\end{proof}

\end{document}